\documentclass[11pt,oneside]{amsart}
\usepackage{amssymb}

\numberwithin{equation}{section}

\newtheorem{thm}{Theorem}[section]
\newtheorem{theorem}{Theorem}
\newtheorem{lemma}[thm]{Lemma}

\newtheorem{corollary}[thm]{Corollary}

\theoremstyle{definition}
\newtheorem{definition}[thm]{Definition}
\newtheorem{notation}[thm]{Notation}

\theoremstyle{remark}
\newtheorem{remark}[thm]{Remark}
\newtheorem*{remark*}{Remark}

\newcommand{\R}{\mathbb R}

\newcommand{\De}{\Delta}

\newcommand{\ep}{\varepsilon}

\newcommand{\co}{\colon}
\newcommand{\pd}{\partial}
\renewcommand{\phi}{\varphi}
\renewcommand{\L}{\mathcal L}

\renewcommand{\tilde}{\widetilde}

\DeclareMathOperator{\grad}{grad}

\newcommand{\be}{\begin{equation}}
\newcommand{\ee}{\end{equation}}

\title[Boundary distance, lens maps and entropy]
{Boundary distance, lens maps and entropy of geodesic 
flows of Finsler metrics}

\author{Dmitri Burago}                                                          
\address{Dmitri Burago: Pennsylvania State University,                          
Department of Mathematics, University Park, PA 16802, USA}                      
\email{burago@math.psu.edu}                                                     
                                                                                
\author{Sergei Ivanov}
\address{Sergei Ivanov:
St.Petersburg Department of Steklov Mathematical Institute,
Russian Academy of Sciences,
Fontanka 27, St.Petersburg 191023, Russia}
\email{svivanov@pdmi.ras.ru}

\thanks{The first author was partially supported
by NSF grant DMS-1205997.
The second author was supported by 
Russian Foundation for Basic Research grant 14-01-00062.}

\subjclass{53C60, 37A35, 37J40}
\keywords{Finsler metric, boundary distance,
lens map, scattering relation, Hamiltonian flow,
perturbation, metric entropy}

\begin{document}

\begin{abstract}
We show that a small perturbation of the 
boundary distance function of a simple Finsler metric
on the $n$-disc
is also the boundary distance function of some Finsler metric.
(Simple metric form an open class containing all flat metrics.)
The lens map is map that sends the exit vector to the 
entry vector as a geodesic crosses the disc. 
We show that a small  perturbation of a lens map 
of a simple Finsler metric is in 
its turn the lens map of some Finsler metric.  
We use this result to construct a smooth
perturbation of the metric on the standard 4-dimensional
sphere to produce positive metric entropy
of the geodesic flow. 
Furthermore, this flow exhibits local generation
of metric entropy, 
that is positive entropy is generated in
arbitrarily small tubes around one trajectory. 
\end{abstract}

\maketitle

\section{Introduction}
\label{sec:intro}

In this paper we prove three theorems. 
At first glance, one of them may seem unrelated to the others, 
however it heavily depends on them. 
Rigorous formulations can be found below. 
Here we give a very informal description to provide the reader
with intuition. 

We deal with the following situation. We have a disc $D=D^n$.
This disc is equipped with a Finsler metric.
For a geometer, Finsler metrics generalize Riemannian ones.
As in the Riemannian case they allow to measure lengths of
tangent vectors. However the assumption that this length
comes from a quadratic form on the tangent space at every point
is dropped.
For people from dynamical systems or classical mechanics
Finsler metrics are quadratically homogeneous Lagrangians.
Thus a Finsler metric determines two structures:
a distance function on $D\times D$ and
a Lagrangian flow on $TD$.

We work with a class of ``nice'' Finsler metrics
called \textit{simple}.
The precise definition can be found below.
To get basic intuition, the reader may think of 
small smooth perturbations of a flat metric on a
Euclidean ball.

Given a simple Finsler metric,
we have a number of related objects:

First of all, we have distances between boundary 
points of the disc. This is
a function on $\pd D\times\pd D$
called the 
\textit{boundary distance function}
associated to the metric inside.

The first question that arises is: 
if we perturb the boundary distance function, can we find a metric
inside which realizes it? 

For Riemannian metrics, the answer is \textbf{no}.
There is a rather obscure obstruction relying
on the Besicovitch inequality \cite{Bes}.
Indeed, consider the standard Euclidean unit square.
Slightly decrease boundary distances between points
close to the opposite vertices. The distances between
the pairs of opposite sides remains~1, therefore
in the Riemannian case the area must be at least~1
by the Besicovitch inequality. On the other hand,
the main result of \cite{Iv13} implies a slight decrease
of the boundary distance function of a flat metric
results in a decrease of area. This shows that this
boundary distance function cannot be realized by
a Riemannian metric.

So far, the authors do not know if a small perturbation 
of the boundary distance function of a Riemannian metric 
in a neighborhood of a single pair of points in $\pd D\times\pd D$
can be  realized by a Riemannian metric. 
Here we ignore all distances between pairs of points outside 
this neighborhood. 
This seems to be a very intriguing problem.

We are particularly interested in this problem since our 
dynamical application requires a perturbation
localized near one geodesic. 
We do not know how to make it Riemannian, nor how to do that in low
dimension.

We can however prove that a sufficiently small
perturbation of the boundary
distance function of a simple Finsler 
metric is the boundary distance function of 
some Finsler metric. Moreover this Finsler metric
is a small perturbation of the original one.
See Theorem \ref{t:main} for the formulation.

The theorem contains two parts.
First we work with Finsler metrics without assuming reversibility.
The second part, which requires additional effort,
provides a construction respecting reversibility.
Reversibility of a Finsler metrics means that its
Finsler norms are symmetric.
Geometers often prefer to assume this symmetry 
for it gives rise to a usual metric space 
with a symmetric distance function.
In dynamics, there are many situations when 
one wants to look at non-reversible Finsler metrics. 
They arise in physics (e.g.\ magnetic fields) and correspond 
to non-reversible Lagrangian flows. 
Non-reversible Finsler metrics determine distance functions
that enjoy all axioms of distance other than the symmetry one. 

\medskip

Next we proceed to the lens map
(sometimes also called Poincare map
or scattering relation).
At this point we suggest the reader to think about
our disc equipped with a thin collar and some smooth 
extension of the metric to that collar.

Our Finsler metric determines a Lagrangian
flow whose Lagrangian is given by the square of
the Finsler norm. Thus we can look at trajectories.
They enter our disc. There are entry and exit points 
for each trajectory of the flow. Of course,
the flow is defined on the (unit) tangent bundle, 
hence the entry and exit points are tangent
vectors.

We would rather turn to the  Hamiltonian language, 
in which case we think about the entry 
and exit co-vectors, 
identified with tangent vectors by the Legendre transform.
Thus we get a scattering relation which sends 
the entry co-vector to the exit co-vector. 
It is referred to as the \textit{dual lens map}. 
This map is symplectic.

There is a very clear relationship between 
the dual lens map and the boundary 
distance function of a simple metric.
As discussed below, they uniquely determine
each other. Namely the derivative of the boundary
distance function with respect to its
first and second arguments at $p,q\in\pd D$ 
is essentially the same thing as the
entry and exit co-vectors corresponding to
the geodesic segment $[pq]$.
This follows from the formula for
the derivative of the distance function,
see \eqref{e:firstvar} and \eqref{e:firstvar1}.

Our second result Theorem \ref{t:lens}
shows that if one perturbs 
the dual lens map of a simple
Finsler metric in the class of symplectic maps, 
we also get the dual lens map of a simple 
Finsler metric.

\medskip

We hope that these results are of some interest on their own.
We also use them to give
partial answers to some classic problems in dynamical 
systems and entropy theory. Here
the word ``entropy" by default means 
the measure-theoretic or ``metric" entropy. 

In our final result Theorem \ref{t:entropy}
we show that the Riemannian metric of the standard 
4-dimensional sphere can be perturbed (in $C^\infty$)
in such a way that the geodesic flow on the unit
tangent bundle is of positive entropy.
This perturbation is in the class of Finsler metrics,
or, equivalently,
in the class of Lagrangian systems. 

The perturbation takes place in a small tube
around one trajectory of
the non-perturbed periodic flow.
This allows us to give an example in which
positive metric entropy is generated 
in arbitrarily small neighborhoods of
one periodic trajectory, see Corollary \ref{c:nonexpansive}.
There is no fixed terminology however many authors
call this property \textit{entropy non-expansiveness}.
This property has been of interest of a number
of mathematicians including 
Bowen, Knieper, Newhouse, the first author,
see e.g.\ \cite{Bowen}, \cite{Bu88}, \cite{Newhouse}.
In the literature this notion was mostly discussed
for topological entropy. Obviously local generation
of metric entropy implies that for topological entropy.

\medskip

Now we proceed with precise definitions and formulations.

\subsection*{Boundary distance}
Recall that a Finsler metric on a smooth manifold $M$
is a function $\phi\co TM\to\R_+$ which is smooth
outside the zero section and such that its
restriction on each fiber $T_xM$, $x\in M$,
is a Minkowski norm. The latter means that
the function
$\phi|_{T_xM}$ is positively homogeneous of degree~1,
positive outside zero, and its square is 
quadratically convex.

The word ``norm'' here is used in a slightly more
general sense than in functional analysis,
namely we do not assume that it is symmetric.
Finsler metrics whose norms are symmetric
(i.e., $\phi(-v)=\phi(v)$ for all $v\in M$)
are said to be \textit{reversible}.
Whenever we assume reversibily,
we mention this explicitly.

The value $\phi(v)$ is interpreted as the length
of a tangent vector $v\in TM$. Then one defines
the length of a smooth path and a
distance function $d_\phi\co M\times M\to\R$.
Similarly to norms, we consider non-symmetric 
distance functions.
By definition, the distance $d_\phi(x,y)$
between points $x,y\in M$ is the infimum of $\phi$-lengths
paths starting at $x$ and ending at~$y$.
Note that $d_\phi$ satisfies the
triangle inequality:
$$
 d_\phi(x,y)+d_\phi(y,z) \ge d_\phi(x,z) .
$$
The distance function $d_\phi$ is symmetric if and only of
$\phi$ is reversible.

The \textit{boundary distance function} $bd_\phi$ of $\phi$
is the restriction of $d_\phi$
to $\pd M\times\pd M$.
That is, $bd_\phi(x,y)$ is the length of
a $\phi$-shortest path in $M$ connecting
the boundary points $x$ and~$y$. 

\textit{Geodesics} in a Finsler manifold are smooth
curves which are locally shortest paths.
>From dynamical viewpoint, geodesics are trajectories
of a Lagrangian system whose Lagrangian is $\frac12\phi^2$.
We only consider geodesics parametrized by arc length.

Let $D=D^n$ denote the $n$-dimensional disc.
A Finsler metric $\phi$ on $D$ is called \textit{simple} if
the following conditions hold:
\begin{enumerate}
\item Every pair of points in $D$ is connected 
 by a unique geodesic.
\item Geodesics depend smoothly on their endpoints.
\item The boundary is strictly convex, that is,
geodesics never touch it at their interior points.
\end{enumerate}
These conditions imply that
all geodesics are minimizing and have no conjugate points.

Our main result Theorem \ref{t:main} asserts
that every $C^\infty$-small perturbation
of the boundary distance function $bd_\phi$ can be
realized as the boundary distance function of a $C^\infty$-small
perturbation of $\phi$.

Note that the distance function is not smooth 
at the diagonal of $D\times D$. 
To avoid silly technicalities, we work in the following set-up.
We only consider perturbations of $bd_\phi$ which are
identical in a neighborhood of the diagonal $\Delta$
of $\pd D\times\pd D$,
and perturbations of $\phi$ identical near the boundary.

\begin{theorem}
\label{t:main}
Let $\phi$ be a simple Finsler metric on $D=D^n$, $f=bd_\phi$,
and $U_0$ a neighborhood of the diagonal $\De\subset\pd D\times\pd D$.
Then for every function $\tilde f\co\pd D\times\pd D\to\R$
such that $\tilde f|_{U_0}=f|_{U_0}$ and
$\tilde f$ is sufficiently $C^\infty$-close to $f$
on $\pd D\times\pd D\setminus U_0$
there exists a simple Finsler metric 
$\tilde\phi$ such that 
$bd_{\tilde\phi}=\tilde f$
and $\tilde\phi$ coincides with $\phi$ 
in a neighborhood of 
$\pd D$.

The choice of $\tilde\phi$ can be made in such a way that
$\tilde\phi$ converges to $\phi$
whenever $\tilde f$ converges to $f$
(in $C^\infty$).

In addition, if $\phi$ is reversible and $\tilde f$ is symmetric,
then $\tilde\phi$ can be chosen reversible.
\end{theorem}

%
%


\subsection*{Lens map}
Let $\phi$ be a Finsler metric on $D$.
We denote by $UTD$ the unit sphere bundle of $\phi$.
By $U_{in}$ we denote the set of unit vectors tangent to $D$
at the boundary $\pd D$ and such that they point inwards.
Analogously $U_{out}$ is the set of unit tangent vectors
at the boundary pointing outwards.

For a vector $v\in U_{in}$, we can look at the geodesic
with the initial vector $v$. As it hits the boundary again,
we get its velocity vector $\beta(v)\in U_{out}$.
This defines a map $\beta\co U_{in}\to U_{out}$.
This map is the \textit{lens map} of $\phi$
(sometimes also called \textit{Poincare map}
or \textit{scattering relation}).
If $\phi$ is reversible, then the lens map $\beta$
is reversible in the following sense:
$-\beta(-\beta(v))=v$ for every $v\in U_{in}$.

Now we proceed to our next Theorem \ref{t:lens}.
It answers the question: which perturbations of $\beta$
are realized as lens maps of Finsler metrics?
The answer is that the perturbed map should be symplectic.
To give a reasonably precise formulation we however need 
more definitions and notation.

First, to employ the symplectic structure, we need, as usual, 
to pass to the cotangent bundle $T^*D$.
Let $\phi^*\colon T^*D\to\R$ be the fiber-wise
dual norm to~$\phi$. That is, for $x\in D$
and $\alpha\in T^*_xD$, one defines
$$
 \phi^*(\alpha) = \sup\{ \alpha(v) \mid v\in UT_xM \} .
$$
We denote by $UT^*D$ the unit sphere bundle
of $\phi^*$.

Let $\L\co TD\to T^*D$ be the Legendre transform 
of the Lagrangian $\frac12\phi^2$.
Due to homogeneity of $\phi$, $\L$ is norm-preserving.
In particular it induces a diffeomorphism between $UTD$ and $UT^*D$.
For a tangent vector $v\in UT_xD$, its Legendre transform
$\L(v)$ is the unique co-vector $\alpha\in UT^*_xD$ 
such that $\alpha(v)=1$.

Then we introduce the dual lens map.  
Namely consider subsets $U^*_{in}=\L(U_{in})$ and
$U^*_{out}=\L(U_{out})$ of $UT^*D$.
The \textit{dual lens map} of $\phi$
is the map $\sigma\co U^*_{in}\to U^*_{out}$
defined by $\sigma=\L\circ\beta\circ\L^{-1}$
where $\beta$ is the lens map of $\phi$.

Note that $U_{in}^*$ and $U_{out}^*$ are $(2n-2)$-dimensional
submanifolds of $T^*D$. Consider the
restriction of the canonical symplectic 2-form of $T^*D$ to 
$U_{in}^*$ and $U_{out}^*$.
This restriction is non-degenerate and hence symplectic.
Indeed, recall that the geodesic flow of $\phi$
is Hamiltonian when considered on $T^*D$.
The non-degeneracy claimed above follows from the fact that
$U_{in}^*$ and $U_{out}^*$
lie in the energy level of this Hamiltonian flow
and transverse to this flow.
For the same reason the dual lens map
$\sigma$ is symplectic.

If $\phi$ is reversible then its dual lens map
is symmetric in the following sense:
$-\sigma(-\sigma(\alpha))=\alpha$
for all $\alpha\in U^*_{in}$.
Note that reversibility of $\phi$ implies that
$-\alpha\in U^*_{out}$ if $\alpha\in U^*_{in}$
and vice versa.

\begin{theorem}
\label{t:lens}
Assume that $n\ge 3$. 
Let $\phi$ be a simple metric on $D=D^n$
and $\sigma$ its dual lens map.
Let $W$ be the complement of a compact set
in $U^*_{in}$.

Then every sufficiently small symplectic perturbation
$\tilde\sigma$ of $\sigma$ such that
$\tilde\sigma|_W^{}=\sigma|_W^{}$
is realized by the dual lens map of
a simple metric $\tilde\phi$ which coincides with $\phi$ in some
neighborhood of $\partial D$.

The choice of $\tilde\phi$ can be made in such a way that
$\tilde\phi$ converges to $\phi$
whenever $\tilde\sigma$ converges to $\sigma$
(in $C^\infty$).

In addition, if $\phi$ is a reversible Finsler metric and
$\tilde\sigma$ is symmetric
in the sense that 
$-\tilde\sigma(-\tilde\sigma(\alpha))=\alpha$
for all $\alpha\in U^*_{in}$, then $\tilde\phi$
can be chosen reversible as well.
\end{theorem}

\begin{remark*}
The domain $U^*_{in}$ and range $U^*_{out}$
of the dual lens map depend on the metric $\phi$.
This is not an issue in Theorem \ref{t:lens}
because $\tilde\phi$ is required to coincide with $\phi$
in a neighborhood of $\pd D$,
and the sets $U^*_{in}$ and $U^*_{out}$ are determined
by the restriction of $\phi$ to $TD|_{\pd D}$.
\end{remark*}

\begin{remark*}
The assumption that $n\ge 3$ in Theorem \ref{t:lens}
is essential. In dimension~2 there is a single obstruction,
namely the map $\tilde\sigma$ should satisfy
certain integral identity, see Remark \ref{r:n=2}.
\end{remark*}

\subsection*{Positive entropy}
The previous two results are applied to prove the following:

\begin{theorem}
\label{t:entropy}
The standard metric of $S^4$ can be perturbed in the
class of reversible Finsler metrics so that the resulting metric
has positive metric entropy of its geodesic flow.
The perturbation can be chosen to be arbitrarily small in $C^\infty$.
\end{theorem}

Note that the analogue of this result for topological entropy
is well-known, see \cite{P93} and \cite{KW94}.

\medskip
\textit{Acknowledgement}. 
We are grateful to Leonid Polterovich for interesting discussions.
In particular he suggested an alternative plan
for the final part of the proof of Theorem \ref{t:entropy}.
His approach is purely symplectic. (We do not use it in this paper.)
We are grateful to Federico Rodriguez Hertz, Gerhard Knieper 
and Sheldon Newhouse
for inspiring discussions and providing us with references.

\section{Preliminaries}
\label{sec:prelim}

The following preliminaries are copied 
with minor modifications from \cite{Iv13}.

Let $M=(M^n,\phi)$ be a Finsler manifold,
possibly with boundary.
In this paper we only consider the case when $M$ is either
the disc $D=D^n$ with a simple metric $\phi$
or an open subset of such a disc.

We say that a continuous function $f\colon M\to\R$
is \textit{distance-like} (with respect to $\phi$)
if $f$ is smooth on $M\setminus\pd M$ and
$\phi^*(d_xf)=1$ for all $x\in M$.
Note that every distance-like function $f$ is forward 1-Lipschitz:
$$
 f(y)-f(x) \le d_\phi(x,y)
$$
for all $x,y\in M$.

If $M=D$ and $\phi$ is a simple metric, then
for every $p\in\pd D$ 
the functions $d_\phi(p,\cdot)$ and $-d_\phi(\cdot,p)$
are distance-like.
Note that $d_\phi(\cdot,p)$ is \textit{not} distance-like
in general (unless the metric is reversible).

The Finslerian \textit{gradient} of a distance-like function
$f\co M\to\R$ at $x\in M$,
denoted by $\grad_\phi f(x)$,
is the unique
tangent vector $v\in U_xTM$ such that $d_xf(v)=1$.
Equivalently, $\grad_\phi f(x)=\L^{-1}(d_xf)$ 
where $\L$ is the Legendre
transform determined by $\phi$.
Beware that, unlike in the Riemannian case,
$\grad_\phi$ is a non-linear operator.

We introduce the following notation 
for a simple metric $\phi$ on $D$.
For distinct points $x,y\in D$, the unique arc length geodesic segment
starting at $x$ and ending at $y$
is denoted by $[xy]$.
By $v_{xy}$ and $w_{xy}$ we denote the velocity vectors
of this geodesic segment at $x$ and $y$, resp.
Thus $v_{xy}\in UT_xD$ and $w_{xy}\in UT_yD$.

Observe that 
the gradient $\grad_\phi d_\phi(x,\cdot)$ at $y$
equals $w_{xy}$.
Indeed, the directional derivative of $d_\phi(x,\cdot)$
along $w_{xy}$ equals~1. Thus the derivative
of the function $d_\phi(x,\cdot)$ at $y$ is given by
\be
\label{e:firstvar}
 d_y d_\phi(x,\cdot) = \L(w_{xy}) .
\ee
A similar argument for the distance-like function
$-d_\phi(\cdot,y)$ shows that
\be
\label{e:firstvar1}
 d_x d_\phi(\cdot,y) = -\L(v_{xy}) .
\ee
One can see that these formulas are special cases 
of the First Variation Formula
for Finslerian length.

\begin{definition}[cf.~{\cite{Iv13}}]
\label{d-env}
Let $M=(M^n,\phi)$ be a Finsler manifold and
$S$ a smooth manifold diffeomorphic to $S^{n-1}$.
We say that a continuous function
$F\co S\times M\to\R$ is an \textit{enveloping function}
for $\phi$ if $F$ is smooth outside $S\times\pd M$
and the following two conditions are satisfied.

(i) For every $p\in S$, the function $F_p:=F(p,\cdot)$
is distance-like;

(ii) For every $x\in M\setminus\pd M$, the map $p\mapsto d_xF_p$
is a diffeomorphism from $S$ to  $UT^*_xD$.
\end{definition}

Note that an enveloping function uniquely determines the metric.
Indeed, at every point $x\in M\setminus\pd M$
the unit co-tangent sphere $UT^*_xM$ can be recovered from $F$
as the image of $S$ under the map $p\mapsto d_xF_p$.
This unit sphere determines the dual norm $\phi^*|_{T_x^*M}$ and
hence the original norm $\phi|_{T_xM}$.

We need the notion of enveloping function for the following
situation.
Let $M=D$, $S=\pd D$ and assume that $\phi$ is simple.
Then the function $F\co S\times D\to\R$ given by
$F(p,x)=d_\phi(p,x)$
is an enveloping function for $\phi$.

The following lemma characterizes enveloping functions 
of this type.

\begin{lemma}\label{l:env-is-distance}
Let $F\co \pd D\times D\to\R$ be an enveloping function for $\phi$
such that the following holds.

(i) $F(p,p)=0$ for every $p\in\pd D$;

(ii) For every distinct $p,q\in\pd D$,
the function $F_p=F(p,\cdot)$ is smooth at $q$ and
$\grad_\phi F_p(q)$ points outwards~$D$.

\noindent
Then $F(p,x)=d_\phi(p,x)$ for all $p\in\pd D$ and $x\in D$.
\end{lemma}

\begin{proof}
Let $p\in\pd D$ and $x\in D\setminus\pd D$.
Consider the backward gradient curve of $F_p$ through $x$,
namely
let $\gamma\co(-t_0,0]\to D$ be the
maximal backward solution of the ODE 
$\dot\gamma(t)=\grad_\phi F_p(\gamma(t))$
with the initial condition $\gamma(0)=x$.

Since $F_p$ is distance-like, we have $\phi(\dot\gamma(t))=1$
and $\frac d{dt} F_p(\gamma(t))=1$ for all~$t$.
Hence $\gamma$ is parametrized by arc length
and $F_p(\gamma(t))=F_p(x)+t$ for all $t\in(-t_0,0]$.
Thus for every $t\le 0$ we have
$$
F_p(x)-F_p(\gamma(t))=-t
=\operatorname{length}_\phi(\gamma|_{[-t,0]}) \ge d_\phi(\gamma(t),x)
$$
for all $t\in(-t_0,0]$. On the hand, we have
$F_p(x)-F_p(\gamma(t))\le d_\phi(\gamma(t),x)$
since $F_p$ is forward 1-Lipschitz. Therefore
\be
\label{e:calibrator}
 F_p(x)-F_p(\gamma(t))=d_\phi(\gamma(t),x)
\ee
for all $t\in(-t_0,0]$.

Since $F_p$ is bounded, \eqref{e:calibrator} implies
that $\gamma$ is not trapped in $D$ and 
therefore eventually hits the boundary.
Due to (ii), the only point where it may hit
the boundary is~$p$.
Thus $\gamma(-t_0)=p$
and then \eqref{e:calibrator} for $t=-t_0$
implies that 
$$
F_p(x) = F_p(x)-F_p(p) = d_\phi(p,x) .
$$
The lemma follows.
\end{proof}


The key point of the proof of Theorem \ref{t:main}
is that a small perturbation of an enveloping function
is again an enveloping function for another Finsler metric.
The precise statement that we need is the following lemma
(borrowed from \cite{Iv13} with minor modifications).

\begin{lemma}\label{l:env-perturbation}
Let $\phi$ be a simple metric on $D$
and $F\co S\times D\to\R$ an enveloping function
for $\phi$.
Let $U\subset D$ be an open set separated from $\pd D$.
Then every function $\tilde F\co S\times U\to\R$
which is sufficiently $C^\infty$-close
to $F|_{S\times U}$
is an enveloping function of some 
Finsler metric $\tilde\phi$ on $U$.

Furthermore $\tilde\phi$ tends to $\phi|_U$ in $C^\infty$
as $\tilde F$ goes to $F$.
\end{lemma}

\begin{proof}
Note that the derivatives of $F$ are bounded
on $S\times U$ since $U$ is contained in
a compact subset of $D\setminus\pd D$.
For every $x\in U$, the map 
$p\mapsto d_xF_p$ from $S$ to $T^*_xD$ 
parametrizes $UT^*_xD$ which is a quadratically 
convex surface in $T^*_xD$.
If $\tilde F$ is close to $F$ then 
the map $p\mapsto d_x\tilde F_p$
is close to $p\mapsto d_xF_p$
and hence it also parametrizes
a quadratically convex surface
close to the original one.
This surface is the unit sphere of some
Minkowski norm on $T^*_xD$.
The dual norm on $T_xD$ is the desired metric 
$\tilde\phi$ at~$x$.
The entire construction is continuous with respect
to relevant $C^\infty$ topologies.
\end{proof}


\section{Proof of Theorem \ref{t:main}}

Let $n\ge 2$, $D=D^n$ and $S=\pd D$.
Let $\phi$, $f$ and $\tilde f$ be as in Theorem \ref{t:main}.
Recall that $f$ is the boundary distance function of $\phi$
and $\tilde f$ a perturbation of $f$ supported outside
a neighborhood $U_0$ of the diagonal $\De\subset S\times S$.

For $r>0$ and $p\in D$, 
let $B_r(p)$ denote the forward $d_\phi$-ball of radius $r$
centered at~$p$, namely $B_r(p)=\{x\in D: d_\phi(p,x)<r\}$.
By $U_r(S)$ we denote the forward $r$-neighborhood
of the boundary, that is,
$U_r(S)=\bigcup_{p\in S} B_r(p)$.

Since $\tilde f$ coincides with $f=bd_\phi$
in a neighborhood $U_0$ of the diagonal,
there exists $\ep>0$ be such that $\tilde f(x,y)=f(x,y)=d_\phi(x,y)$ 
for all $x,y\in S$ with $d_\phi(x,y)<5\ep$.
Fix fix this $\ep$ for the rest of this section.

Define a function $F\co S\times D\to\R$ 
by $F(p,x)=d_\phi(p,x)$.
Then $F$ is an enveloping function for $\phi$.
In order prove Theorem \ref{t:main},
we construct 
an enveloping function $\tilde F\co S\times D\to\R$
for the desired metric $\tilde\phi$.
We make $\tilde F$ out of $F$ and a function
$G$ which takes care of a neighborhood of the boundary
(see Lemma \ref{l:collar} below).

Fix $p\in S$ and define 
a function $H_p\co D\to\R$ by
\begin{equation}
\label{e:H}
 H_p(x) = \max_{y\in S} \{ \tilde f(p,y) - d_\phi(x,y) \} .
\end{equation}
Note that, if $\tilde f=f$ then $H_p(x)=d_\phi(p,x)$
and the maximum in \eqref{e:H} is attained at 
the (unique) point $y\in S$ such that $x\in[py]$.

Though $H_p$ is defined on the whole disc $D$,
we are going to restrict it to the set $U_\delta(S)\setminus B_\ep(p)$
where $\delta$ is a sufficiently small positive constant
(depending on $\phi$ and $\ep$).

For every $q\in S\setminus\{p\}$, we define a tangent
vector $\tilde w_{pq}\in UT_qD$ pointing outwards $D$,
as follows. 
Consider the derivative $d_q\tilde f(p,\cdot)\in T^*_qS$.
If $\tilde f=f$, the $\phi$-norm of this derivative
is strictly less than~1. This follows from \eqref{e:firstvar}
and the fact that the geodesic $[pq]$ is transverse to the boundary.
Hence the $\phi^*$-norm of this derivative is less than~1
whenever $\tilde f$ is sufficiently close to $f$ (in $C^\infty$,
on $S\times S\setminus U_0$).
Hence there exists 
a unique vector $w\in UT_qD$ pointing outwards $D$ and
such that $\L(w)|_{T_qS}=d_q\tilde f(p,\cdot)$.
We take this vector $w$ for $\tilde w_{pq}$.
Note that $\tilde w_{pq}$ 
depends smoothly on $p$, $q$ and $\tilde f$.

In the case when $\tilde f=f$ we have $\tilde w_{pq}=w_{pq}$ 
by \eqref{e:firstvar} where $w_{pq}$ it the velocity vector
of the geodesic $[pq]$ at $q$
(see the notation introduced in Section \ref{sec:prelim}).
Therefore $\tilde w_{pq}$ tends to $w_{pq}$ 
(in $C^\infty$, as a function of $p$ and $q$),
and $\tilde w_{pq}=w_{pq}$ if $d_\phi(p,q)<5\ep$.

For every distinct $p,q\in S$ 
and $t\ge 0$ 
define
$\gamma_{pq}(-t)=\exp_q(-t\tilde w_{pq})$
where
$\exp_q$ is the exponential map of the metric $\phi$ at $q$.
That is, $\gamma_{pq}$ is a geodesic which
hits the
boundary at $q$ with velocity $\tilde w_{pq}$.
This geodesic is parametrized so that $q=\gamma_{pq}(0)$.
The domain of $\gamma_{pq}$ is an interval of the form
$[-T,0]$ where $T$ depends smoothly on $p$, $q$ and $\tilde f$.
Hence $T\ge \min\{d_\phi(p,q),\ep\}$ provided that $\tilde f$
is sufficiently close to $f$.

In the case when $\tilde f=f$ we have $\gamma_{pq}=[pq]$.
Since $\tilde w_{pq}$ is a $C^\infty$-small perturbation of $w_{pq}$,
it follows that 
the geodesics $\{\gamma_{pq}\}_{q\in S\setminus\{p\}}$
can intersect only within the ball $B_\ep(p)$
and cover the set $D\setminus B_\ep(p)$.

\begin{lemma}
\label{l:Hp}
Let $p,q\in S$ and $t\ge 0$ be such that 
$\gamma(-t)\in D\setminus B_\ep(p)$.
Then
\be
\label{e:Hp}
 H_p(\gamma_{pq}(-t)) = \tilde f(p,q)-t ,
\ee
provided that $\tilde f$ is sufficiently close to $f$.
\end{lemma}

\begin{proof}
Let $x=\gamma_{pq}(-t)$.
Then $d_\phi(x,q)=t$.
Let $y\in S$ 
be a point of maximum in \eqref{e:H}.
Since the derivative at a point of maximum is zero,
we have
$$
 d_y (\tilde f(p,\cdot) - d_\phi(x,\cdot))|_{T_yS} = 0.
$$
Hence
\begin{equation}
\label{e:xi}
 d_y\tilde f(p,\cdot) = d_yd_\phi(x,\cdot)|_{T_yS}=\L(w_{xy})|_{T_yS}
\end{equation}
where the second equality follows from \eqref{e:firstvar}.
This implies that $w_{xy}=\tilde w_{py}$,
therefore $x$ belongs to the geodesic $\gamma_{py}$.
This implies that $y=q$ because otherwise the geodesics
$\gamma_{pq}$ and $\gamma_{py}$ do not intersect outside $B_\ep(p)$.

Thus $q$ is a point of maximum in \eqref{e:H}. 
Now \eqref{e:H} takes the form
$$
 H_p(x) = \tilde f(p,q) - d_\phi(x,q) = \tilde f(p,q) - t
$$
and the lemma follows.
\end{proof}

Since for a fixed $p$ the geodesics $\gamma_{pq}$
cover $D\setminus B_\ep(p)$,
the identity \eqref{e:Hp} uniquely determines 
the function $H_p$ in $D\setminus B_\ep(p)$.
Moreover, the map $(q,t)\mapsto\gamma_{pq}(-t)$
is a diffeomorphism from an appropriate subset
of $S\times\R_+$ to $D\setminus B_\ep(p)$.
Hence \eqref{e:Hp} implies that $H_p$
is smooth on $D\setminus B_\ep(p)$.

Now let us show that $H_p|_{D\setminus B_\ep(p)}$
is a distance-like function.
First observe that $H_p$  is forward 1-Lipschitz.
Indeed, $H_p$ is the supremum of functions
of the form $x\mapsto \tilde f(p,y)-d_\phi(x,y)$,
each of which is forward 1-Lipschitz
by the triangle inequality for $d_\phi$.
Next, \eqref{e:Hp} implies that $H_p$
grows at unit rate along each geodesic $\gamma_{pq}$
within $D\setminus B_\ep(p)$. 
Hence $H_p|_{D\setminus B_\ep(p)}$ is distance-like
and
\be
\label{e:gradHp}
 \grad_\phi H_p(\gamma_{pq}(-t))=\dot\gamma_{pq}(-t)
\ee
whenever $\gamma_{pq}(-t)\in D\setminus B_\ep(p)$.

\begin{lemma}
\label{l:collar}
There exist a neighborhood $V\subset D$ of $S$
and a function $G\co S\times V\to\R$ such that

1. $G$ is an enveloping function for $\phi|_V^{}$.

2. $G$ coincides with $\tilde f$ on $S\times S$.

3. $G$ coincides with $F$ in a neighbourhood
of $\Delta$ in $S\times V$.
(Recall that $\Delta$ is the diagonal of $S\times S$.)

4. For every two different points $p,q\in S$,
the gradient of the function $G_p=G(p,\cdot)$ at $q$
points outwards $D$.

Furthermore, $G$ is constructed out of $\phi$ and $\tilde f$
in such a way that $G$ converges to $F|_{S\times V}$ (in $C^\infty$)
as $\tilde f$ goes to $f$.
\end{lemma}

\begin{proof}
Since the boundary is strictly convex with respect to $\phi$, 
there exists 
$\delta=\delta(\phi,\ep)>0$
such that the following holds. If $p,q\in S$, 
$x\in [pq]\cap U_{\delta}(S)$
and $\ep\le d_\phi(p,x)\le 2\ep$,
then $d_\phi(p,q)<3\ep$.
Let $V=U_\delta(S)$.

For each $p\in S$ define a function $G_p\co V\to\R$ by
\begin{equation}
\label{e:Gp}
 G_p(x) = 
 \begin{cases}
   H_p(x) &\quad\text{if $d_\phi(p,x)>\ep$}, \\
   d_\phi(p,x) &\quad\text{if $d_\phi(p,x)<2\ep$} .
 \end{cases}
\end{equation}
The domains in the two cases in \eqref{e:Gp}
overlap, yet the two formulas yield the same value.
Indeed, let $x\in V$ be such that $\ep<d_\phi(p,x)<2\ep$
and let $q\in S$ be such that $x\in[pq]$.
Then $d_\phi(p,q)<3\ep$ by the choice of $\delta$.
Then $\tilde w_{pq}=w_{pq}$ and 
hence $\gamma_{pq}=[pq]\ni x$. 
By Lemma \ref{l:Hp}
it follows that
$$
 H_p(x) = \tilde f(p,q)-d_\phi(x,q) = 
 d_\phi(p,q)-d_\phi(x,q) = d_\phi(p,x) .
$$
Thus the two formulas in \eqref{e:Gp}
agree on the overlap.
Hence \eqref{e:Gp} defines a continuous function $G_p\co V\to\R$.

Define $G\co S\times V\to\R$ by $G(p,x)=G_p(x)$.
Clearly $G$ is smooth outside $\Delta$
and $G$ tends to $F|_{S\times V}$ as $\tilde f$ goes to~$f$.
We are going to show that $G$ satisfies the assertions
of the lemma.
The third assertion is trivial by construction.

To prove the second assertion, consider $p,q\in S$.
If $d_\phi(p,q)<2\ep$ then $G(p,q)=d_\phi(p,q)=\tilde f(p,q)$
by \eqref{e:Gp} and the choice of~$\ep$.
In the other case, namely if $d_\phi(p,q)\ge 2\ep$,
we have $G(p,q)=H_p(q)=\tilde f(p,q)$
by setting $t=0$ in Lemma \ref{l:Hp}.

The fourth assertion follows from the fact
that $\grad_\phi G_p(q)=\tilde w_{pq}$
for all distinct $p,q\in S$.
In the case $d_\phi(p,q)>\ep$ this fact
follows from \eqref{e:gradHp},
and in the case $d_\phi(p,q)\le\ep$
we have $\grad_\phi G_p=w_{pq}=\tilde w_{pq}$. 

It remains to verify the first assertion of the lemma.
First observe that $G_p$ is distance-like because both $H_p$
and $d_\phi(p,\cdot)$ are distance-like on~$V$.
Let $x\in V\setminus S$. 
Since $G_p$ is distance-like,
we have a smooth map
$p\mapsto d_xG_p$ from $S$ to $UT^*_xD$.
It remains to prove that this map is a diffeomorphism.
Recall that $G$ is a $C^\infty$-small perturbation 
of $F|_{S\times V}$ and 
the map $p\mapsto d_xF_p$ is a diffeomorphism
from $S$ to $UT^*_xD$.
Furthermore $F$ and $G$ coincide
within the set $\De_{2\ep}=\{(p,x):|px|<2\ep\}$.
Hence $d_xG_p=d_xF_p$ for $(p,x)\in\De_{2\ep}$.
Outside $\De_{2\ep}$
the derivatives of $F$ are uniformly bounded.
Hence a sufficiently small perturbation
supported outside $\De_{2\ep}$
yields a diffeomorphism
$p\mapsto d_xG_p$ from $S$ to $UT^*_xD$.
\end{proof}

\begin{proof}[Proof of Theorem \ref{t:main}]
Let $G\co S\times V\to\R$ be a function
constructed in Lemma \ref{l:collar}.
We glue the desired function $\tilde F$ out of $F$ and $G$
using a partition of unity.
Let $h\co D\to[0,1]$ be a smooth function
such that $h=1$ outside $V$ and $h=0$ in a neighborhood 
$V_0$ of $S$.
Define $\tilde F\co S\times D\to\R$ by
$$
 \tilde F(p,x) = h(x) F(p,x) + (1-h(x)) G(p,x) .
$$
Then $\tilde F$ is an enveloping function
for a Finsler metric $\tilde\phi$ which
coincides with $\phi$ within $V_0$.
Indeed, $\tilde F=G$ in $V_0\subset V$
and hence Lemma \ref{l:collar} implies that $\tilde F$
is an enveloping function for $\phi$ within $V_0$.
Outside $V_0$, $\tilde F$ converges to $F$ (in $C^\infty)$
as $\tilde f\to f$. Hence by Lemma \ref{l:env-perturbation}
it is an enveloping function of some Finsler metric
$\tilde\phi$ which is close to $\phi$.

Since $\tilde F$ is an extension of $\tilde f$,
Lemma \ref{l:env-is-distance} implies that
$\tilde f$ is the boundary distance function of $\tilde\phi$.
This finishes the proof of Theorem \ref{t:main}
in the non-reversible case.

In order to make $\tilde\phi$ reversible
(provided that $\phi$ is reversible
and $\tilde f$ is symmetric)
we need some preparations.

\begin{notation}
\label{n:Gamma}
For a set $U\subset S\times S$ denote by $\Gamma(U)$
the set of velocity vectors of all geodesics $[pq]$ such
that $(p,q)\in U$.
\end{notation}

\begin{remark}
\label{r:locality}
The construction of $\tilde\phi$
also guarantees the following
property which we need below.
If $U\subset S\times S$ is an open set
such that $\tilde f|_U^{}=f|_U^{}$,
then $\tilde\phi$ coincides with $\phi$ on $\Gamma(U)$.

Indeed, let $(p,q)\in U$ and $x\in[pq]$.
The argument in Step~2 of the proof
of Lemma \ref{l:collar} shows that every local maximum
in \eqref{e:H} is global. Hence $H_p(x)=F(p,x)$.
Going through the subsequent parts of the construction
one sees that $\tilde F(p,x)=F(p,x)$.
Since $U$ is open, this argument also shows
that $\tilde F(p',x)=F(p',x)$ for all $p'$
from a neighborhood of $p$ in $S$. 
Therefore $\tilde\phi^*$ coincides with $\phi^*$
in a neighborhood of the co-vector $d_xF_p$.
Let $v$ be the velocity vector of $[pq]$ at $x$.
Then $d_xF_p$ is the Legendre transform of~$v$.
Hence $\tilde\phi(v)=\phi(v)$ as claimed.
\end{remark}

Now we are in position to prove the last
assertion of Theorem \ref{t:main}.
Let $R$ denote the involution of $S\times S$
permuting its arguments, i.e., $R(p,q)=(q,p)$.
Recall that $\tilde f$ and $f$ coincide in
a neighborhood $U_0$ of the diagonal $\Delta\subset S\times S$.
Fix an open covering $\{U_i\}_{i=1}^N$ 
of $S\times S\setminus U_0$
such that for every $i$ 
the sets $U_i$ and $R(U_i)$ 
have disjoint closures.

Connect $f$ to $\tilde f$ by a sequence
of functions $f=f_0,f_1,\dots,f_N=\tilde f$
where every $f_i\co S\times S\to\R$
is a symmetric function obtained from $f_{i-1}$
by a $C^\infty$-small perturbation supported
in $U_i\cup R(U_i)$.
By induction, we construct a sequence of
reversible metrics $\phi_i$ on $D$ such that
$f_i$ is the boundary distance function of $\phi_i$
for each $i$.
The induction step goes as follows.
First apply Theorem \ref{t:main}
to $\phi_{i-1}$ in place of $\phi$ and $f_i$ in place of $\tilde f$,
resp. This yields a non-reversible metric $\tilde\phi_i$ whose
boundary distance function is $f_i$.
By Remark \ref{r:locality}, $\tilde\phi_i$ coincides with $\phi_{i-i}$
outside $\Gamma(U_i\cup R(U_i))$ where $\Gamma(\dots)$ is 
defined as in Notation \ref{n:Gamma} however with respect 
to $\phi_{i-1}$ in place of $\phi$.
Since $\Gamma(U_i)$ and $\Gamma(R(U_i))$ have disjoint closures,
we can change the metric within $\Gamma(R(U_i))$ 
so that it becomes reversible. 
Namely we define the desired metric $\phi_i$ by
$$
 \phi_i(v) =
 \begin{cases}
   \tilde\phi_i(-v) &\quad \text{if } v\in \Gamma(R(U_i)) , \\
   \tilde\phi_i(v) &\quad \text{otherwise},
 \end{cases}
$$
for every $v\in UTD$.
This metric has the same boundary distance function.
Indeed, if $(p,q)\notin R(U_i)$ then the geodesic $[pq]$
of the metric $\tilde\phi_i$ remains a geodesic as we
replace $\tilde\phi_i$ by $\phi_i$.
Hence $d_{\phi_i}(p,q)=d_{\tilde\phi_i}(p,q)=f_i(p,q)$
if $(p,q)\notin R(U_i)$.
This identity holds for $(p,q)\in R(U_i)$ 
because in this case $(q,p)\notin R(U_i)$
and the distance is symmetric.

It remains to verify that the set of functions $\tilde f$
reachable by sequences $f_0,\dots,f_N$ such that the
construction works, contains a neighborhood of $f$ 
in our space of boundary distance functions.
(Recall that this is the space functions
on $S\times S$ coinciding with $f$ in $U_0$,
equipped with the $C^\infty$ topology.)
The covering $\{U_i\}$ and hence the number of steps
in the construction it determined by $U_0$.
The construction 
which produces $\tilde\phi$ out of $\phi$ and $\tilde f$
is explicit and hence continuous in its
arguments $\phi$ and $\tilde f$ (regarded
as elements of respective functional spaces).
Furthermore it is defined for an open set
of pairs $(\phi,\tilde f)$.
The decomposition and symmetrization procedures above
also enjoy similar continuity properties.
Therefore if $\tilde f$ is chosen from a suitable
neighborhood of $bd_\phi$ then each pair $(\phi_{i-1},f_i)$
in the inductive construction belongs to the domain
where Theorem~1 applies.
\end{proof}

\section{Proof of Theorem \ref{t:lens}}

As in the previous sections,
let $\phi$ be a simple metric on $D=D^n$
and $S=\pd D$. 
Recall that in Theorem \ref{t:lens}
we assume that $n\ge 3$.
The metric $\phi$ (or its restriction to
a neighborhood of $S$) determines
the sets $U^*_{in},U^*_{out}\subset T^*D$
as in Section \ref{sec:intro}.

For the proof of the theorem we need a zoo of notations. 
The good news is that they have clear geometric meaning.

We think of trajectories of a Hamiltonian flow
on the unit co-tangent bundle $UT^*D$. 
They enter $UT^*D$ and leave it.
Thus we have an entry co-vector in $U^*_{in}$
and an exit co-vector in $U^*_{out}$
for each trajectory. 
Of course each of them
can be viewed as a pair consisting
of the base point in $\pd D$
and the co-vector component.

A diffeomorphism $\sigma\co U^*_{in}\to U^*_{out}$
tells us the exit co-vector from
the entry one.
In the following notations and constructions
we use only certain properties of $\sigma$
and do not assume that it arose from a flow.
We apply these constructions to both
$\sigma$ and $\tilde\sigma$.

First we need a map
$P_\sigma$ which takes an entry co-vector 
and produces two points $p,q\in\pd D$ 
which are the entry and exit base points of the 
``trajectory'' of $\sigma$ determined by that co-vector.
Analogously, $Q_\sigma$ takes an exit co-vector and produces 
the entry and exit base points.

Formally,
we define maps $P_\sigma\co U^*_{in}\to S\times S$
and $Q_\sigma\co U^*_{out}\to S\times S$ by
$$
 P_\sigma(\alpha) = (\pi(\alpha),\pi(\sigma(\alpha))),
 \qquad\alpha\in U^*_{in},
$$
and
$$
 Q_\sigma(\alpha) = (\pi(\sigma^{-1}(\alpha)),\pi(\alpha)),
 \qquad\alpha\in U^*_{out},
$$
where $\pi\co T^*D\to D$ is the bundle projection.
Note that $P_\sigma=Q_\sigma\circ\sigma$.

Next we need the inverse maps $P_\sigma^{-1}$ and $Q_\sigma^{-1}$.
The map  $P_\sigma^{-1}$ takes two distinct points from 
the boundary $\pd D$ as its input and tells
us the entry co-vector.
Similarly, $Q_\sigma^{-1}$ takes two distinct boundary
points and reports the exit co-vector. 
In order to ensure that $P_\sigma^{-1}$ and $Q_\sigma^{-1}$
exist we assume that $\sigma$ is nice in the following sense:

\begin{definition}
\label{d:nice}
We say that a map $\sigma\co U^*_{in}\to U^*_{out}$ 
is \textit{nice} if the following conditions
are satisfied:
\begin{enumerate}
\item 
$\sigma$ is a symplectic diffeomorphism between $U^*_{in}$
and $U^*_{out}$.
\item
$P_\sigma$ is a diffeomorphism between $U^*_{in}$
and $S\times S\setminus\Delta$.
\item
$Q_\sigma$ is a diffeomorphism between $U^*_{out}$
and $S\times S\setminus\Delta$.
\end{enumerate}
Here, as in the previous sections, $\Delta$
denotes the diagonal of $S\times S$.
\end{definition}

In fact, the third condition 
in Definition \ref{d:nice}
follows from the second one
and the identity $P_\sigma=Q_\sigma\circ\sigma$.
We include them both in the definition
for the clarity of exposition.
The assumption that $\sigma$ is symplectic
is crucial. Nonetheless it is used only in 
one place, namely in the proof of Lemma \ref{l:lambdaclosed}.

If $\sigma$ is the dual lens map of a simple metric then
$\sigma$ is nice.
Furthermore if $\tilde\sigma$ is 
obtained from a nice map $\sigma$ by
a sufficiently
small compactly supported perturbation 
(as in Theorem \ref{t:lens}),
then $\tilde\sigma$ is nice as well.

The last bit of notation is 1-form $\lambda_\sigma$ on 
$S\times S\setminus\Delta$ defined as follows.
For $p,q\in S$, $p\ne q$, $\xi\in T_pS$, $\eta\in T_qS$,
define
\be
\label{e:lambdasigma}
 \lambda_\sigma(\xi,\eta) 
 = - P_\sigma^{-1}(p,q)(\xi) + Q_\sigma^{-1}(p,q)(\eta) .
\ee
Here $(\xi,\eta)$ is regarded as an element of 
$T_{(p,q)}(S\times S\setminus\Delta)$ through
the standard identification
$T(S\times S) \simeq TS\times TS$.
The maps $P_\sigma^{-1}$ and $Q_\sigma^{-1}$
in \eqref{e:lambdasigma}
are correctly defined since $\sigma$ is nice.
Moreover $P_\sigma^{-1}(p,q)\in (U^*_{in})^{}_p$
and $Q_\sigma^{-1}(p,q)\in (U^*_{out})^{}_q$
where $(U^*_{in})^{}_p$ and $(U^*_{out})^{}_q$
are the fibers of $U^*_{in}$ and $U^*_{out}$
over $p$ and $q$, resp. Therefore the terms
$P_\sigma^{-1}(p,q)(\xi)$ and $Q_\sigma^{-1}(p,q)(\eta)$
in \eqref{e:lambdasigma} make sense.

The following lemma explains where the definition 
of $\lambda_\sigma$ comes from. 
The lemma tells us that if $\sigma$ arises from a simple metric 
then $\lambda_\sigma$ is the derivative of 
the distance function, as follows:

\begin{lemma}
\label{l:lambda=df}
Let $f=bd_\phi$ and let $\sigma\co U^*_{in}\to U^*_{out}$
be a nice map. Then
$\sigma$ is the dual lens map of $\phi$
if and only if $df=\lambda_\sigma$
on $S\times S\setminus\Delta$.
\end{lemma}

\begin{proof}
We begin with the ``only if'' part.
Let $\sigma$ be the dual lens map of $\phi$.
Let $p,q\in S$, $p\ne q$,
$\alpha=\L(v_{pq})$ and $\beta=\L(w_{pq})$
where $v_{pq}$ and $w_{pq}$ are the velocity vectors
of the geodesic $[pq]$ at $p$ and $q$,
see notation in Section \ref{sec:prelim}.
Then $\beta=\sigma(\alpha)$ and
$P_\sigma(\alpha)=Q_\sigma(\beta)=(p,q)$.
By \eqref{e:firstvar1} and \eqref{e:firstvar} we have
$d_p f(\cdot,q) = -\alpha$
and 
$d_q f(p,\cdot) = \beta$.
Hence for every $\xi\in T_p S$ and $\eta\in T_q S$,
$$
 df(\xi,\eta) = -\alpha(\xi) + \beta(\eta)
 = - P_\sigma^{-1}(p,q)(\xi) + Q_\sigma^{-1}(p,q)(\eta)
 = \lambda_\sigma(\xi,\eta) .
$$
This proves the ``only if'' part of the lemma.

To prove the ``if'' part, let $\widehat\sigma$ be
the dual lens map of $\phi$.
By the ``only if'' part we have $df=\lambda_{\widehat\sigma}$
on $S\times S\setminus\Delta$.
Thus it suffices to prove that $\sigma=\widehat\sigma$
provided that $\lambda_\sigma=\lambda_{\widehat\sigma}$.

Let $\alpha\in U^*_{in}$, $p=\pi(\alpha)\in S$,
$\beta=\widehat\sigma(\alpha)\in U^*_{out}$
and $q=\pi(\beta)\in S\setminus\{p\}$.
Then $P_{\widehat\sigma}(\alpha)=Q_{\widehat\sigma}(\beta)=(p,q)$.
Since $\lambda_{\widehat\sigma}=\lambda_\sigma$,
substituting $\eta=0$ into \eqref{e:lambdasigma}
yields that
$$
 P_\sigma^{-1}(p,q)(\xi) = P_{\widehat\sigma}^{-1}(p,q)(\xi)
 =\alpha(\xi)
$$
for every $\xi\in T_pS$.
That is, $P_\sigma^{-1}(p,q)|_{T_pS}=\alpha|_{T_pS}$.
Since $P_\sigma^{-1}(p,q)\in (U^*_{in})^{}_p$
and a co-vector from $(U^*_{in})^{}_p$ is uniquely determined
by its restriction to $T_pS$, it follows that
$P_\sigma^{-1}(p,q)=\alpha$.
Hence $\sigma(\alpha)\in (U^*_{out})^{}_q$.

Similarly, substituting $\xi=0$ into \eqref{e:lambdasigma}
yields that $Q_\sigma^{-1}(p,q)|_{T_qS}=\beta|_{T_qS}$
and therefore $Q_\sigma^{-1}(p,q)=\beta$.
Hence $\sigma^{-1}(\beta)\in (U^*_{in})^{}_p$.
Thus
\be
\label{e:lambda=df1}
 \pi(\sigma^{-1}(\beta)) = p = \pi(\alpha)
 =\pi(\sigma^{-1}(\sigma(\alpha)))
\ee

Since $\sigma$ is nice, $Q_\sigma$ is a diffeomorphism
from $U^*_{out}$ to $S\times S\setminus\Delta$.
In particular, the restriction $Q_\sigma|_{(U^*_{out})^{}_q}$ 
is a diffeomorphism
from $(U^*_{out})^{}_q$ to $(S\setminus\{q\})\times\{q\}$.
Hence the first coordinate map of this restriction, 
namely
$\pi\circ\sigma^{-1}|_{(U^*_{out})^{}_q}$, is injective.
Since $\beta\in (U^*_{out})_q$ and $\sigma(\alpha)\in (U^*_{out})_q$,
this injectivity and \eqref{e:lambda=df1} imply that 
$\sigma(\alpha)=\beta$.
Since $\alpha$ is an arbitrary element of $U^*_{in}$
and $\beta=\widehat\sigma(\alpha)$,
it follows that $\sigma=\widehat\sigma$.
This finishes the proof of the ``if'' part of the lemma.
\end{proof}

\begin{lemma}
\label{l:lambdaclosed}
$\lambda_\sigma$ is a closed 1-form
for every nice map $\sigma\co U^*_{in}\to U^*_{out}$.
\end{lemma}

\begin{proof}
Let $\nu$ denote the canonical 1-form on $T^*D$.
Recall that the canonical 1-form is defined as follows:
for $\alpha\in T^*D$ and $\tau\in T_\alpha T^*D$,
$$
 \nu(\tau) = \alpha (d\pi(\tau))
$$
where $\pi\co T^*D\to D$ is the bundle projection.
The two terms in \eqref{e:lambdasigma} are actually
pull-backs of $\nu$ by $P_\sigma^{-1}$ and $Q_\sigma^{-1}$, 
namely
$$
 \lambda_\sigma 
 = - (P_\sigma^{-1})^* \nu + (Q_\sigma^{-1})^* \nu .
$$
Hence
$$
 d\lambda_\sigma 
 = - (P_\sigma^{-1})^* (d\nu) + (Q_\sigma^{-1})^* (d\nu) .
$$
Let $\omega$ be the canonical symplectic 
form on $T^*D$.
Recall that $\omega=d\nu$.
Therefore
$$
\begin{aligned}
 d\lambda_\sigma 
 &= - (P_\sigma^{-1})^* \omega + (Q_\sigma^{-1})^* \omega 
 = - (P_\sigma^{-1})^* \omega + (\sigma\circ P_\sigma^{-1})^* \omega \\
 &=- (P_\sigma^{-1})^* \omega + (P_\sigma^{-1})^* (\sigma^*\omega) = 0
\end{aligned}
$$
Here we use the identity $P_\sigma=Q_\sigma\circ\sigma$
and the fact that $\sigma^*\omega=\omega$
because $\sigma$ is symplectic.
Thus $d\lambda_\sigma=0$ and the lemma follows.
\end{proof}

\begin{proof}[Proof of Theorem \ref{t:lens}]
Now let $\sigma$ be the dual lens map of $\phi$
and $\tilde\sigma$ be a perturbation of $\sigma$
as in Theorem \ref{t:lens}. Recall that $\tilde\sigma$
is nice provided that the perturbation is sufficiently small.
Then by Lemma \ref{l:lambdaclosed} the 1-form
$\lambda_{\tilde\sigma}$ is closed.
Since $n\ge 3$, the space $S\times S\setminus\Delta$
is simply connected and therefore $\lambda_{\tilde\sigma}$
is exact.
Hence there exists a smooth function
$\tilde f\co S\times S\setminus\Delta\to\R$
such that $d\tilde f=\lambda_{\tilde\sigma}$.

Since $\tilde\sigma$ coincide with $\sigma$
outside a compact set, 
$\lambda_{\tilde\sigma}$ and $\lambda_\sigma$
also coincide outside a compact subset of
$S\times S\setminus\Delta$. 
That is, there exists a tubular 
neighborhood $U$ of $\Delta$
in $S\times S$ such that 
$\lambda_{\tilde\sigma}|_{U\setminus\Delta}
=\lambda_\sigma|_{U\setminus\Delta}$.
This fact and Lemma \ref{l:lambda=df}
imply that 
\be\label{e:df=dtildef}
d\tilde f|_{U\setminus\Delta}
=df|_{U\setminus\Delta}
\ee
where $f$ is the boundary distance function of $\phi$.
Since $n\ge 3$, the set $U\setminus\Delta$ is connected.
This and \eqref{e:df=dtildef} imply that $\tilde f-f$
is constant on $U\setminus\Delta$.
Since $\tilde f$ is defined up to
an additive constant,
we may assume that $\tilde f=f$ on $U\setminus\Delta$.
Now we can extend $\tilde f$ to the whole $S\times S$
by setting it zero on $\Delta$.
Then $\tilde f$ is a small perturbation of $f$
in the sense of Theorem \ref{t:main}.
Hence by Theorem~\ref{t:main} there exists
a simple Finsler metric $\tilde\phi$ whose boundary
distance function is~$\tilde f$.
Moreover $\tilde\phi$ coincides with $\phi$
in a neighborhood of the boundary
and is close to $\phi$ in $C^\infty$.
Now Lemma \ref{l:lambda=df} applied to $\tilde\phi$
in place of $\phi$ implies that $\tilde\sigma$
is the dual lens map of $\sigma$.

It remains to prove the last assertion of Theorem \ref{t:lens},
namely that $\tilde\phi$ can be chosen reversible
if $\phi$ is reversible and 
$\tilde\sigma$ satisfies the symmetry condition:
$-\tilde\sigma(-\tilde\sigma(\alpha))=\alpha$
for all $\alpha\in U^*_{in}$.
By the last assertion of Theorem \ref{t:main}, 
it suffices to verify that $\tilde f$ is symmetric,
that is, $\tilde f(p,q)=\tilde f(q,p)$
for all $p,q\in S$. 
First observe that the symmetry condition
on $\tilde\sigma$ 
implies that
\be\label{e:Psymmetry}
  P_{\tilde\sigma}^{-1}(q,p)=-Q_{\tilde\sigma}^{-1}(p,q)
\ee
for all $p,q\in S$, $p\ne q$.
Indeed, if $\alpha=P_{\tilde\sigma}^{-1}(p,q)$
and
$\beta=Q_{\tilde\sigma}^{-1}(p,q)$
then $\tilde\sigma(\alpha)=\beta$,
hence $\tilde\sigma(-\beta)=-\alpha$ by the symmetry condition.
Therefore $P_{\tilde\sigma}^{-1}(q,p)=-\beta$
and \eqref{e:Psymmetry} follows.

Now consider the function $\tilde g\co S\times S\to\R$
defined by $\tilde g(x,y)=\tilde f(y,x)$.
Let $p,q\in S$, $p\ne q$, $\xi\in T_pS$ and $\eta\in T_qS$.
Then
$$
\begin{aligned}
d_{(p,q)}\tilde g(\xi,\eta) = d_{(q,p)}\tilde f(\eta,\xi)
&= -P_{\tilde\sigma}^{-1}(q,p)(\eta)+Q_{\tilde\sigma}^{-1}(q,p)(\xi) \\
&= Q_{\tilde\sigma}^{-1}(p,q)(\eta)-P_{\tilde\sigma}^{-1}(p,q)(\xi)
=d_{(p,q)}\tilde f(\xi,\eta) .
\end{aligned}
$$ 
Here we used the identity $d\tilde f=\lambda_{\tilde\sigma}$,
the definition of $\lambda_{\tilde\sigma}$ and
\eqref{e:Psymmetry}.
Thus $d\tilde g=d\tilde f$, hence $\tilde g=\tilde f$, that is,
$\tilde f$ is symmetric.
This finishes the proof of Theorem \ref{t:lens}.
\end{proof}

\begin{remark}
\label{r:n=2}
The above argument does not work in dimension $n=2$
for two reasons. 
The first issue is that $S\times S\setminus\Delta$
is not simply connected so the closed 1-form $\lambda_{\tilde\sigma}$
a priori may not be exact.
However it is easy to see that this 1-form is exact
even in dimension~2. Indeed, the fundamental group
of $S\times S\setminus\Delta$ is cyclic and its generator
can be realized by a loop arbitrarily close to $\Delta$.
Since $\lambda_{\tilde\sigma}$ and $\lambda_\sigma$
coincide in a neighborhood of $\Delta$, their integrals
over this loop are equal. Since $\lambda_\sigma=df$
is exact, these integrals are zero. Hence $\lambda_{\tilde\sigma}$
is exact as well.

The second issue with $n=2$ is that $U\setminus\Delta$
is not connected, where $U$ is a tubular neighborhood
of $\Delta$ is $S\times S$. In fact, $U\setminus\Delta$
consists of two components. These components are
represented by points $(p,q_1)$ and $(p,q_2)$ in 
$S\times S\setminus\Delta$ where $p\in S$ is an
arbitrary point and $q_1,q_2\in S$ are close to $p$
and lie on different sides of it.
It may happen that the function $\tilde f$
(an antiderivative of $\lambda_{\tilde\sigma}$)
approaches different values as we tend to the
diagonal $\Delta$ from different sides.
In this case $\tilde f$ cannot be extended to
the whole $S\times S$ as a continuous function.
To rule out this situation we need to impose
an additional condition on $\tilde\sigma$, namely
\be\label{e:integral}
 \int_{\{p\}\times (S\setminus\{p\})} \lambda_{\tilde\sigma} = 0 
\ee
for some (and then all) $p\in S$.
This condition is a necessary and sufficient one
and it is easy to give examples when it is not satisfied.

Taking into account the definition of $\lambda_{\tilde\sigma}$,
the condition \eqref{e:integral} can be interpreted as follows.
For a fixed $p\in S$, the image $\tilde\sigma(U^*_{in}|^{}_p)$
is a graph of a differential 1-form on $S\setminus\{p\}$,
and \eqref{e:integral} is equivalent to the requirement
that the (improper) integral of this 1-form over
$S\setminus\{p\}$ is zero.
\end{remark}

\section{Proof of Theorem \ref{t:entropy}}

In this section we deduce Theorem \ref{t:entropy}
from Theorem \ref{t:lens}.
Consider the unit disc $D^6\subset\R^6$ equipped with
the standard symplectic structure.
We need the following lemma.

\begin{lemma}
\label{l:D6}
There exists a symplectomorphism $\theta\co D^6\to D^6$
which is arbitrarily close to the identity in $C^\infty$,
coincides with the identity map near the boundary,
and has positive metric entropy.
\end{lemma}

\begin{proof}
Let $M$ be a surface of genus $\ge 2$
equipped with a Riemannian metric of negative curvature.
Let $I_1=(-1,1)$, $I_2=(-2,2)$ and
$N=M\times I_2$.
Let $B^*M$ denote the bundle of unit balls
in $T^*M$.

First we construct a compactly supported
Hamiltonian $H\co T^*N\to\R$ whose Hamiltonian flow
has positive metric entropy.
Let $\psi\co\R\to\R_+$ be a smooth cut-off function
which equals 1 on $[-1,1]$ and vanishes
outside $(-\frac32,\frac32)$.
For a co-vector $(x,y,\xi,\eta)\in T^*N$
where $x\in M$, $y\in I_2$, $\xi\in T^*_xM$,
$\eta\in T^*_y I_2$, we define
$$
 H(x,y,\xi,\eta) = 
 \psi(y)\cdot\psi(|\xi|)\cdot\psi(|\eta|)
 \cdot \frac{|\xi|^2}2
$$
where $|\xi|$ is the norm of $\xi$
determined by the Riemannian metric of~$M$.

Observe that $H$ is compactly supported.
Consider the Hamiltonian flow of $H$
restricted to the set 
$$
B^*M\times B^*I_1\subset T^*M\times T^*I_2= T^*N .
$$
On this set we have $ H(x,y,\xi,\eta) =|\xi|^2/2$.
Recall that the standard Hamiltonian of our 
Riemannian metric is given by $H_{Riem}(x,\xi) =|\xi|^2/2$.
Hence the Hamiltonian flow of $H$ restricted to
$B^*M\times B^*I_1$
is the product of the geodesic flow
of our Riemannian metric (restricted to $B^*M$)
and the trivial flow on $B^*I_1$.
Hence the metric entropy of this flow is positive.

Now let $U\subset\R^3$ be an open set diffeomorphic
to $N$ (namely $U$ is a thickened handlebody).
A diffeomorphism between $N$ and $U$ induces
a symplectomorphism between $T^*N$ and $T^*U$.
This symplectomorphism sends the Hamiltonian
flow of $H$ to a compactly supported
Hamiltonian flow on $T^*U\subset T^*\R^3\simeq\R^6$.
The support of this flow fits into $D^6$
upon a suitable rescaling.

Thus there exists a Hamiltonian flow $\{\Phi^t\}$ on $D^6$
which vanishes near the boundary and has positive
metric entropy. 
Finally, in order to construct the 
desired symplectomorphism $\theta$,
we take a map $\Phi^t$ for a sufficiently small $t>0$.
\end{proof}

\begin{proof}[Proof of Theorem \ref{t:entropy}]
Consider the sphere $S^4$ with its standard Riemannian metric
and let $D\subset S^4$ be a small geodesic ball.
Then $D$ is diffeomorphic to $D^4$ and the metric of $S^4$
restricts to a simple Riemannian metric on~$D$.
Using the notations introduced before Theorem \ref{t:lens},
consider the dual lens map $\sigma\co U^*_{in}\to U^*_{out}$.

Fix a co-vector $\alpha\in U^*_{in}$ and let $\omega=\beta^*(\alpha)$.
Let $V\Subset U^*_{in}$ be a neighborhood of $\alpha$
in $U^*_{in}$.
We choose $V$ to be symplectomorphic to a rescaled copy
of the interior of the standard disc $D^6$.
By Lemma \ref{l:D6} there exists a compactly supported
symplectomorphism $\theta\co V\to V$ with positive metric
entropy and arbitrarily close to the identity map in
the $C^\infty$ topology.
We extend $\theta$ to a map from the entire $U^*_{in}$
to itself so that it is identity on $U^*_{in}\setminus V$.
We use the same letter $\theta$ for this extension.

Define a perturbed dual lens map 
$\tilde\sigma\co U^*_{in}\to U^*_{out}$
by $\tilde\sigma = \sigma\circ\theta$.
If $\theta$ is sufficiently close to identity,
then by Theorem \ref{t:lens} the map $\tilde\sigma$
can be realized as the dual lens map of
a Finsler metric $\tilde\phi$. 
Furthermore $\tilde\phi$ can be chosen close
to the original Riemannian metric and coinciding
with it in a neighborhood of the boundary of~$D$.
We paste the metric $\tilde\phi$ on $D$ with
the original metric on $S^4\setminus D$
and obtain a smooth Finsler metric on $S^4$.
This metric is the desired perturbation with positive entropy.

Indeed, since the original geodesic flow is periodic,
the return map of the perturbed geodesic flow equals
$\theta$ on~$V$ (and remains the identity elsewhere).
Thus the perturbed flow contains an invariant open set
where the flow is a suspension of $\theta$.
Hence the perturbed geodesic flow has positive
metric entropy. 
\end{proof}

The disc $D$ in the proof of Theorem \ref{t:entropy}
can be chosen arbitrarily small. Thus the perturbation
of the geodesic flow is confined to an arbitrarily small
neighborhood of one (periodic) orbit of the original flow.
This allows us to construct an example with local generation
of metric entropy:

\begin{corollary}
\label{c:nonexpansive}
There exists a smooth perturbation
of the standard metric of $S^4$ 
(in the class of reversible Finsler metrics)
such that the following holds.
There is a periodic trajectory $\gamma$ such that
for every $\ep>0$ the $\ep$-neighborhood of $\gamma$
contains an open invariant set in which the metric
entropy is positive.
In particular, the flow is entropy non-expansive
in the sense of \cite{Bowen}.
\end{corollary}

\begin{proof}
Begin with a sequence
$\{\gamma_i\}$ of disjoint geodesics in $S^4$
converging to a geodesic $\gamma$.
For each $i$ choose a neighborhood $U_i$ of $\gamma_i$
in such a way that these neighborhoods are disjoint.
Construct a small perturbation of
the metric within each $U_i$ as in Theorem \ref{t:entropy}.
These perturbations should be so small that
their derivatives go to zero as $i\to\infty$.
Then the union of these perturbations provides the desired example.
\end{proof}

Note that both metric and topological entropy in the 
$\ep$-neighborhoods in Corollary \ref{c:nonexpansive}
must vanish as $\ep\to0$.
This follows from the fact that $\ep$-local entropy 
of any $C^\infty$ flow tends to zero as $\ep\to 0$, see~\cite{Newhouse}.

\bibliographystyle{plain}

\end{document}